\documentclass[11pt,reqno]{amsart}

\usepackage{amsmath,amsthm,amssymb,color,soul}
\usepackage[tmargin=1.2in,bmargin=1.2in,rmargin=1.2in,lmargin=1.2in]{geometry}
\usepackage[breaklinks=true]{hyperref}
\usepackage{verbatim}
\theoremstyle{plain}

\newtheorem{theorem}{Theorem}[section]

\newtheorem{corollary}[theorem]{Corollary}
\newtheorem{proposition}[theorem]{Proposition}
\newtheorem{lemma}[theorem]{Lemma}

\theoremstyle{definition}

\newtheorem{example}[theorem]{Example}
\newtheorem{remark}[theorem]{Remark}

\theoremstyle{remark}

\numberwithin{equation}{section}

\begin{document}
\title{Maps preserving the sum-to-difference ratio}
\date{\today}

\author{Sunil Chebolu}
\address[S.~Chebolu]{Department of Mathematics, Illinois State University,
Normal, IL 61761, USA}
\email{\tt schebol@ilstu.edu}

\author{Apoorva Khare}  \thanks{ A.K.\ was partially supported by a
Shanti Swarup Bhatnagar Award from CSIR (Govt.\ of India).}
\address[A.~Khare]{Department of Mathematics, Indian Institute of
Science, Bangalore 560012, India}
\email{\tt khare@iisc.ac.in}

\author{Anindya Sen}
\address[A.~Sen]{Accountancy \& Finance Department, 
University of Otago, Dunedin 9016, NZ}
\email{\tt anindya.sen@otago.ac.nz}

\begin{abstract}
For a field $\mathbb{F}$, what are all functions $f \colon \mathbb{F}
\rightarrow \mathbb{F}$ that satisfy the functional equation $f \left(
(x+y)/(x-y) \right) = (f(x) + f(y))/(f(x) - f(y))$ for all $ x \neq y$ in
$\mathbb{F}$? We solve this problem for the fields $\mathbb{Q},
\mathbb{R}$, and a class of its subfields that includes the real
constructible numbers, the real algebraic numbers, and all quadratic
number fields. We also solve it over the complex numbers, and on
any subfield of $\mathbb{R}$, if $f$ is continuous over the reals. The
proofs involve a mix of algebra in all fields, analysis over the real
line, and some topology in the complex plane.
\end{abstract}

\subjclass[2020]{39B22, 39B52, 12F05}

\maketitle
%\tableofcontents

\section{Introduction}

The problem of finding all functions satisfying a given condition has been a major and fruitful endeavor in mathematics. The best-known case is the area of Differential Equations. Here, functions
are typically assumed to be smooth, and the conditions are phrased in terms of the function and its derivatives.
In contrast, the field of ``Functional Equations'' typically involves imposing an algebraic condition on the function and no additional restrictions (or fairly weak ones) on the class of functions to be considered.

A famous example is the problem of finding all functions $f : \mathbb{R} \to \mathbb{R}$ that satisfy the Cauchy functional equation:
\begin{equation}\label{Ecauchy}
f : \mathbb{R} \to \mathbb{R}, \qquad f(x+y) = f(x) + f(y) \ \  \forall \,x,y \in \mathbb{R}.
\end{equation}

This equation, introduced by Cauchy in the early 19th century, is
often regarded as the ``mother of all functional equations,'' as it
inspired many related equations later studied by mathematicians such as
Jensen, d’Alembert, Abel, and others. Their work helped establish
functional equations as a rich branch of mathematics at the crossroads of
analysis, algebra, and number theory. It is worth noting that these
equations also arise naturally in the physical sciences: Jensen’s
equation captures thermodynamic convexity, d’Alembert’s equation
underlies the wave equation and oscillatory motion, and even the additive
Cauchy equation expresses the principle of linear superposition,
fundamental to mechanics, electricity, and Fourier analysis.

It is well known that the only continuous solutions to~\eqref{Ecauchy}
are the linear maps $f(x)=cx $, where $c=f(1)$ -- an instructive exercise
for a freshman math major. However, relaxing the continuity condition
leads to extremely wild\footnote{We quickly mention how to ``construct'' all
solutions, using the Axiom of Choice. Let $\mathcal{B}$ be a Hamel
$\mathbb{Q}$-basis of $\mathbb{R}$ including $1$, and choose any function
$f_0 : \mathcal{B} \to \mathbb{R}$. Then extending $f_0$ to all of
$\mathbb{R}$ by $\mathbb{Q}$-linearity gives a solution, and most of
these are not linear (e.g.\ if $f_0(r) \neq f_0(1) r$ for some $r \in
\mathcal{B} \setminus \{ 1 \}$).}
solutions. For instance, any solution to the Cauchy equation that is
discontinuous at even a single point is unbounded on every interval in
$\mathbb{R}$ and has a graph which is dense in $\mathbb{R}^2$! 

An intriguing aspect of functional equations is the fact that the
solution set can change drastically with an innocuous modification of the
equation. For instance, if we modify the Cauchy functional equation to
$f(x + y) = f(x) - f(y)$, then it is an easy exercise to show that the only
solution is the function which is identically zero!

A parallel small modification of the Cauchy functional equation is
\[
f(x-y) = f(x)-f(y), \ \  \forall\; x,y \in \mathbb{R}.
\]
It is easy to see that this leads us back to the Cauchy equation. (Just
write $x = (x-y)+y$, then rename the variables.)

\subsection*{The problem.}

The above leads us to a natural question: what if we combine both the
additive and subtractive properties into one functional equation? We pose
the problem over an arbitrary field $\mathbb{F}$:

\textit{Determine all maps  $f  \colon  \mathbb{F} \to \mathbb{F}$ that
satisfy the functional equation}
\begin{equation}\label{Esumdiff}
f \left( \frac{x+y}{x-y} \right) = \frac{f(x) + f(y)}{f(x) - f(y)},
\ \  \forall \; x \neq y \in \mathbb{F}.
\end{equation}

In other words, what are all self-maps on $\mathbb{F}$ that preserve the
ratio of sum to difference? For brevity of exposition, we will henceforth
refer to functions that satisfy ~\eqref{Esumdiff} as SD maps -- an
abbreviation for ``Sum-Difference ratio preserving maps.''

In stark contrast to the Cauchy equation, we are able to prove -- without
assuming continuity -- that the only SD map over $\mathbb{R}$ is the
identity map $f(x) \equiv x$. 

In fact, we show that the same result holds for SD maps over
$\mathbb{Q}$, as well as all Euclidean subfields of $\mathbb{R}$, such as
the field of constructible numbers (see just before and after
Theorem~\ref{Euclidean}).

We then extend our solutions to the complex plane, $\mathbb{C}$, in two
different ways, culminating in the result that the only $\mathbb{R}$
preserving SD maps over $\mathbb{C}$ are the identity and conjugate maps.
Here and below, for a subfield $\mathbb{F}$ of $\mathbb{C}$, we say that
a map $f : \mathbb{F} \to \mathbb{F}$ is \textit{$\mathbb{R}$ preserving}
if $f( \mathbb{F} \cap \mathbb{R}) \subseteq \mathbb{F} \cap
\mathbb{R}$.

This striking result suggests a possible connection between SD maps over
any field, $\mathbb{F}$, and the field automorphisms of $\mathbb{F}$ -- a
connection which we investigate over other well-known fields such as
$\mathbb{Q}(i)$, the field of Gaussian rationals. Along the way, we prove
several corollaries imposing further conditions on our solutions. We also
provide counterexamples for some natural conjectures.

Our main results can be summarized in the following theorems.

\begin{theorem}
\label{SDgeneral}
    Let $\mathbb{F}$ be a field and $f \colon \mathbb{F} \to \mathbb{F} $ be an SD map. Then:
\begin{enumerate}
    \item $f$ must be the identity map if $\mathbb{F}$ is $\mathbb{Q}$, $\mathbb{R}$, or any Euclidean subfield of $\mathbb{R}$.
    \item $f$ must be the identity map or the conjugate map if $\mathbb{F}$ is any quadratic extension of $\mathbb{Q}$.
\end{enumerate}
\end{theorem}

\begin{theorem}\label{SDRpreserving}
    Let $\mathbb{F}$ be a subfield of $\mathbb{C}$ and $f \colon \mathbb{F} \to \mathbb{F} $ be an $\mathbb{R}$ preserving SD map. Then:
\begin{enumerate}
    \item $f$ must be the identity map or the conjugate map if $\mathbb{F}$ is $\mathbb{C}$.
    \item $f$ must be the identity map or the conjugate map if $\mathbb{F}$ is the quadratic closure of any Euclidean subfield of $\mathbb{R}$.
\end{enumerate}
\end{theorem} 

\begin{theorem} \label{SDcont}
Let $\mathbb{F}$ be a subfield of $\mathbb{C}$ and $f
\colon \mathbb{F} \to \mathbb{F} $ be a continuous SD map
(see Remark~\ref{Rcont}). Then:
  \begin{enumerate}
      \item $f$ must be the identity map for $\mathbb{F}$ any subfield of $\mathbb{R}$.
      \item $f$ must be the identity or conjugate map for $\mathbb{F}$ any subfield of $\mathbb{C}$ which contains the Gaussian rationals $\mathbb{Q}(i)$.
  \end{enumerate}
\end{theorem}

\begin{remark}\label{Rcont}
Throughout this work, when we discuss continuous SD-maps $: \mathbb{F}
\to \mathbb{F}$, we will assume that $\mathbb{F}$ is a subfield of
$\mathbb{C}$, equipped with the subspace topology inherited from
$\mathbb{C}$.
\end{remark}

Remarkably, all of the statements above hold if we replace the term ``SD
map'' with ``field automorphism''. This prompts the very natural
question: \emph{Are field automorphisms the only solutions to our
functional equation?} We show this is not the case by giving explicit SD
maps that are not field automorphisms. These counterexamples suggest that
our functional equation belongs to a distinct class and, to the best of
our knowledge, has not been previously studied in the literature. 

Our results and methods reveal the rich potential of our functional
equation and connect to several tools and techniques from algebra,
analysis, and even some topology.  We end the paper with some questions
inviting further study.

\vskip 5mm
\noindent 
\textbf{Acknowledgements} The authors wish to thank the anonymous referees for carefully going through the article and for their suggestions, which helped improve the presentation of the article and clarify it further.

\section{Preliminaries}

We begin by proving some properties of SD maps that will be used
throughout the sequel.

\begin{proposition}
Let $\mathbb{F}$ be a field of characteristic zero. Let $f$ be an SD map on $\mathbb{F}$. That is,  $f \colon \mathbb{F} \rightarrow \mathbb{F}$ is such that
\[
f \left( \frac{x+y}{x-y} \right) = \frac{f(x) + f(y)}{f(x) - f(y)},
\qquad \forall \; x \neq y \in \mathbb{F}.
\]
Then we have the following. 
\begin{enumerate}
    \item $f$ is injective.
    \item $f(0) = 0$ and $f(1) = 1$.
     \item $f$ is an odd function: $f(-x) = -f(x)$ for all $x$ in $\mathbb{F}$. 
    \item $f$ is multiplicative: $f(xy) = f(x)f(y)$ for all $x$ and $y$ in $\mathbb{F}$.
   
\end{enumerate}    
\end{proposition}

\begin{proof}
\begin{enumerate}

\item If $x \ne y$ but $f(x) = f(y)$, then the LHS of our functional equation has some value in $\mathbb{F}$, but the RHS is undefined. Hence, $f$ is injective.
    \item Setting $y=0$ and letting $x \ne 0$ in our functional equation, we see that \[f(1) = \frac{f(x)+f(0)}{f(x)-f(0)}.\]
    After simplifying, we get $ f(x)(f(1)-1) = f(0)(f(1)+1)$.
    If $f(1) \ne 1$, then \[f(x) = \frac{f(0)(f(1)+1)}{f(1)-1} \qquad  \forall \; x \ne 0 \in \mathbb{F}.\] 
    But $\mathbb{F}$ contains at least two distinct nonzero values of $x$.
    This means $f(1) =1$ (by injectivity).  Since $f(1)=1$, the equation \[ f(x)(f(1)-1) = f(0)(f(1)+1) = 2f(0)\]
    gives us  $f(0) =0$.
    \item Setting $y=-x$ where $x \ne 0$ in our functional equation gives
    \[ f(0) = 0 = \frac{f(x)+f(-x)}{f(x)-f(-x)}.\]
    This means $f(x) = -f(-x)$ for any non-zero $x$. The equation also holds for $x=0$. This shows that $f$ is odd.
    \item Set $y = kx$, where $k \ne 1$ and $x \ne 0$ in our functional equation gives:
    \begin{align*}
        &\ f \left(\frac{1+k}{1-k}\right) = f\left(\frac{x+kx}{x-kx}\right) = \frac{f(x)+f(kx)}{f(x)-f(kx)}\\
        \Rightarrow &\ f\left( \frac{k+1}{k-1}\right) = \frac{f(kx)+f(x)}{f(kx)-f(x)}  =1 + \frac{2f(x)}{f(kx)-f(x)} = 1 + \frac{2}{\left(\frac{f(kx)}{f(x)} -1\right)} .
    \end{align*}
    
    Notice that the left-hand side is independent of $x$, so the right-hand side must be too. That is, the RHS for general $x$ equals its value for $x=1$. As $f(1)= 1$, we get:
\begin{align*}
&\ \frac{2}{(\frac{f(kx)}{f(x)} -1)} = f\left(\frac{k+1}{k-1}\right) - 1
= \frac{2}{f(k) - 1}\\
\implies &\ \frac{f(kx)}{f(x)} - 1 = f(k) - 1 \implies
f(kx) = f(k)f(x)
\end{align*}
for all $k\neq 1$ and $x \neq 0$.  However, when $x=0$ or $k=1$, this
equation is clearly true. Thus, we get, after relabeling,  $f(xy) =
f(x)f(y)$ for all $x$ and $y$. 
\end{enumerate}
\end{proof}

\begin{remark}
    Note that the results also hold for any field $\mathbb{F}$ with ${\rm char}(\mathbb{F}) \neq 2$. If ${\rm char}(\mathbb{F}) = 2$, then $2 = 0$, so we cannot prove that $f(0) = 0$ or that $f$ is multiplicative. 
\end{remark}

We see from the above proposition that 
the restriction of $f$ to $\mathbb{F}^\times$ -- the multiplicative group of $\mathbb{F}$ -- is an injective group homomorphism. Moreover, since $f$ is an odd function, we also infer that $f(-1) = -1$.
\section{Subfields of $\mathbb{R}$}
Since any field of characteristic $0$ contains $\mathbb{Q}$ as a subfield,  to understand SD maps over other fields such as $\mathbb{R}$, $\mathbb{C}$, and $\mathbb{Q}(i)$ it is important to understand $SD$ maps over $\mathbb{Q}$. The following theorem is the foundation upon which all the subsequent results are built.

\begin{theorem} \label{QRESULT}
    Let $f \colon \mathbb{Q} \rightarrow \mathbb{Q}$ be an SD map. Then $f(x) = x$ for all $x$ in $\mathbb{Q}$.
\end{theorem}

\begin{proof}

We begin by explicitly computing $f$ at the positive integers in terms of
$f(2)$ and via a recurrence relation for $f(n)$. The key observation is
to set $x = n$ and $y = 1$ in~\eqref{Esumdiff}:
\[
f \left( \frac{n+1}{n-1} \right) = \frac{f(n)+1}{f(n)-1}, \qquad \forall \; 
n \ge 2.
\]

Since $f$ is injective and $n \neq 1$, we have $f(n) \neq 1$. By
multiplicativity, we therefore obtain a recurrence relation for the
$f(n)$:
\begin{equation}\label{Erecur}
f(n+1) = f(n-1) \cdot \frac{f(n)+1}{f(n)-1}, \qquad \forall \; n \ge 2.
\end{equation}

Now using this relation, one can explicitly compute $f(n)$ -- in terms of
$u := f(2)$. This is a straightforward verification (left as an exercise
for the reader) which yields:
\begin{align}
f(0) = 0, \quad f(1) = 1, \quad f(2) = u, \quad f(3) = \frac{u+1}{u-1},
\quad f(4) = u^2,\\
f(5) = \frac{u^2+1}{(u-1)^2}, \qquad
f(6) = u(u^2 - u + 1), \\
f(7) = \frac{u^3 - u^2 + u + 1}{u^3 - 3u^2 + 3u - 1}, \qquad
f(8) = u^2(u^2 - 2u + 2).
\end{align}

Since $f$ is multiplicative, $f(8) = f(2)f(4)$, i.e.,
\[
u^2 (u^2 - 2u + 2) = u^3 \quad \Longleftrightarrow \quad u^2 (u-1)(u-2) =
0.
\]
As $u = f(2)$ and $f$ is injective, $u \neq 0,1$. Thus $f(2) = u = 2$.

Now we claim by induction on $n \geq 2$  that $f(n) = n$. The
induction step follows from~\eqref{Erecur}:
\begin{equation}\label{Einduct}
f(n+1) = f(n-1) \cdot \frac{f(n)+1}{f(n)-1} = (n-1) \cdot \frac{n+1}{n-1}
= n+1.
\end{equation}
As $f$ is multiplicative, $f(p/q) = f(p)/f(q) = p/q$ for all positive integers
$p,q$. Since $f$ is odd, it fixes $\mathbb{Q}$.
\end{proof}

Our proof gives much more; it applies to any SD map between fields of characteristic $0$. We record this as a theorem.

\begin{theorem}\label{Tmult}
Let $f$ be an SD map on a field of characteristic $0$. Then $f$ 
fixes $\mathbb{Q}$ pointwise.
\end{theorem}

Another result which follows quickly is:

\begin{corollary}\label{Qcont}
If $\mathbb{F}$ is any subfield of $\mathbb{R}$, the only continuous SD map on $\mathbb{F}$
is the identity.
\end{corollary}

\begin{proof}
By Theorem \ref{Tmult}, $f$ fixes $\mathbb{Q}$ pointwise. Hence by
continuity, it fixes the closure of $\mathbb{Q}$ pointwise. But this
closure is $\mathbb{F}$, since $\mathbb{Q}$ is dense in $\mathbb{R}$ and
so in $\mathbb{F}$.
\end{proof}

\begin{remark}
    Just as the field of real numbers $\mathbb{R}$ is obtained by completing the field of rational numbers $\mathbb{Q}$ with respect to the usual absolute value norm $|\cdot|$, for any given prime $p \geq 2$,  the field of $p$-adic numbers $\mathbb{Q}_p$ is obtained by completing $\mathbb{Q}$ with respect to the $p$-adic norm $|\cdot|_p$; see \cite{padic} for details. The completed field $\mathbb{Q}_p$ contains $\mathbb{Q}$ as a dense subfield. Thus, we get the following corollary, which is analogous to the real case.
\end{remark}

\begin{corollary} \label{Qpcont}
    If $\mathbb{F}$ is any subfield of $\mathbb{Q}_p$, the only continuous SD map on $\mathbb{F}$ is the identity.
\end{corollary}

These corollaries prompt us to ask whether these results are true for any subfield without assuming continuity. The answer is negative, as can be seen by considering the field automorphism $f \colon  \mathbb{Q}(\sqrt{2}) \rightarrow \mathbb{Q}(\sqrt{2})$, where $f$ fixes $\mathbb{Q}$ and $f(\sqrt{2}) = -\sqrt{2}$. It is easy to verify that this is a discontinuous function that satisfies our functional equation. 

This motivates us to consider how far these results extend without assuming continuity.

\begin{theorem}\label{Rmaps}
    Let $f \colon \mathbb{R} \rightarrow \mathbb{R}$ be an SD map. Then $f$ is the identity map.
\end{theorem}

Analyzing the proof shows that the above result holds for several subfields of $\mathbb{R}$.

\begin{proof}
Note that if $x >  0$, then 
\[ f(x) = f((\sqrt{x})^2) = f(\sqrt{x})^2 > 0.\]
Since $f$ is an odd function, we also get that $f(x)<0$ when $x < 0$.

We use this observation to show that $f$ is an increasing function. To this end, let $a > b$.  We have 3 cases to consider. $a > b > 0$, $ 0 > a > b$, and $a > 0 > b$. The last case follows from the above observation because in this case $f(a) > 0$ and $f(b) <0$.  Since $f$ is an odd function, it suffices to establish the result in the first case $a > b > 0$. Then,
\[
f\left( \frac{a+b}{a-b} \right) = \frac{f(a) + f(b)}{f(a) - f(b)} > 0,
\]
since the argument is positive and $f$ maps positives to positives. The numerator of the RHS is positive, so the denominator must also be positive. This shows $f(a) > f(b)$.

We already know from Theorem~\ref{Tmult} that $f$ fixes all rationals. 
So, take an irrational $x$. Suppose $x < f(x) $. Since $\mathbb{Q}$ is dense in $\mathbb{R}$, we can pick a rational $q \in (x,f(x))$. Since $f$ is increasing,  $x < q \implies f(x) < f(q) = q$. But $f(x) > f(q)$, so we have a contradiction. We get a similar contradiction if $f(x) < x$. This shows that $f(x)=x$ for all $x$.
\end{proof}

It is an interesting exercise for the reader to prove the above theorem under the additional assumption that 
$f$ is continuous, but without assuming Theorem \ref{QRESULT}. This special case was, in fact, what initiated this project!

We can extend the above theorem to an important class of subfields of $\mathbb{R}$.
 A field is \emph{Euclidean} if it is ordered and all non-negative elements have a square root. Clearly, any subfield of $\mathbb{R}$ inherits an ordering from $\mathbb{R}$. With this definition in hand, we show the following result.

\begin{theorem}\label{Euclidean}
    Let $\mathbb{F}$ be any Euclidean subfield of $\mathbb{R}$ and $f \colon \mathbb{F} \rightarrow \mathbb{F}$ be an SD map. Then $f$ is the identity map.
\end{theorem}

\begin{proof}
The proof is essentially identical to Theorem \ref{Rmaps}.
If $x>0$ in $\mathbb{F}$, then by definition, $\sqrt{x} \in \mathbb{F}$, which implies $f(x) > 0$.
Hence, we can show that $f$ is strictly increasing on $\mathbb{F}$, just as for Theorem \ref{Rmaps}.
But now $\mathbb{Q}$ is dense in $\mathbb{F}$ and $f$ fixes $\mathbb{Q}$. Hence, by the same argument as for Theorem \ref{Rmaps}, $f(x) = x$ for all $x \in \mathbb{F}$.    
\end{proof}

\begin{remark}
   Note that the proof above does not require $f$ to take values in the Euclidean field $\mathbb{F}$ -- it suffices to have $f$ take values in $\mathbb{R}$. 
\end{remark}

Examples of Euclidean subfields of $\mathbb{R}$ include the Constructible Numbers (these are the reals obtained from $1$ via ruler and compass), which are highly relevant to Euclidean geometry; and the field of Real Algebraic Numbers, which are important in algebraic number theory.

\section{The field of complex numbers}

Our exploration of subfields of $\mathbb{R}$ leads us, quite naturally, to investigate SD maps over $\mathbb{C}$. As a preliminary, we note that our results in the previous section imply the following:

\begin{lemma}\label{Lcont}
    Let $f \colon \mathbb{C} \rightarrow \mathbb{C}$ be an SD map. Then the following are equivalent:
\begin{enumerate}
    \item $f$ preserves $\mathbb{R}$, i.e., $f(\mathbb{R}) \subseteq \mathbb{R}$.
    \item $f$ is continuous on $\mathbb{R}$.
    \item $f$ fixes $\mathbb{R}$ pointwise.
\end{enumerate}    
\end{lemma}

The lemma above indicates the following fact.

\begin{proposition}\label{Cbasics}
    Let $f \colon \mathbb{C} \rightarrow \mathbb{C}$ be an $\mathbb{R}$ preserving SD map and let $\mathbb{C}^\times$ denote the multiplicative group of $\mathbb{C}$. Then one of the following must hold.
    \begin{enumerate}
        \item $f(z/\bar{z}) = z/\bar{z}$ and $f(z^2) = z^2$, for all $z \in \mathbb{C}^\times$.
        \item $f(z/\bar{z}) = \bar{z}/z$ and $f(z^2) = \bar{z}^2$, for all $z \in \mathbb{C}^\times$.
    \end{enumerate}
\end{proposition}
\begin{proof}
    Since $f$ preserves $\mathbb{R}$, $f(x) = x$ for all $x \in  \mathbb{R}$. Furthermore, \[-1 = f(-1) = f(i^2) = f(i)^2 \implies f(i) = \pm{i}.\]
    \begin{enumerate}
        \item Suppose $f(i) = i$ and pick any $z = a + ib$ in $\mathbb{C}^{\times}$ where $a, b \in \mathbb{R}$. 
        
        \noindent Then we have:
        \[f\left(\frac{z}{\bar{z}} \right) = f\left(\frac{a + ib}{a -ib}\right) = \frac{f(a) + f(ib)}{f(a) - f(ib)} = \frac{f(a) + f(i)f(b)}{f(a) - f(i)f(b)} = \frac{a + ib}{a - ib} = \frac{z}{\bar{z}}. \]
        But now note that $|z| \in \mathbb{R} \implies f(|z|) = |z|$. Hence, for all $ z \in \mathbb{C}^\times$, we have
        \[\frac{z}{\bar{z}} = f \left(\frac{z}{\bar{z}} \right) = f\left(\frac{zz}{\bar{z}z} \right) = f\left(\frac{z^2}{|z|^2} \right) = \frac{f(z^2)}{|z|^2} \implies f(z^2) = \frac{z |z|^2}{\bar{z}} = z^2 . \]
        
        \item Suppose $f(i) = -i$ and pick any $z = a + ib$ in $\mathbb{C}^{\times}$ where $a, b \in \mathbb{R}$. 
        
        \noindent Then we have:
        \[f\left(\frac{z}{\bar{z}}\right) = f\left(\frac{a + ib}{a -ib}\right) = \frac{f(a) + f(ib)}{f(a) - f(ib)} = \frac{f(a) + f(i)f(b)}{f(a) - f(i)f(b)} = \frac{a - ib}{a + ib} = \frac{\bar{z}}{z}. \]
        
        \noindent Hence, for all $ z \in \mathbb{C}^\times$, we have
        \[\frac{\bar{z}}{z} = f\left(\frac{z}{\bar{z}}\right) = f\left(\frac{zz}{\bar{z}z}\right) = f\left(\frac{z^2}{|z|^2}\right) = \frac{f(z^2)}{|z|^2} \implies f(z^2) = \frac{\bar{z}|z|^2}{z} = \bar{z}^2.  \qedhere \]
    \end{enumerate}
\end{proof}

We now look at continuous SD maps over $\mathbb{C}$. Note that continuity over $\mathbb{C}$ is a strictly stronger property than $\mathbb{R}$ preserving, which only guarantees continuity over $\mathbb{R}$.

\begin{theorem}\label{Ccont}
  Let $f \colon \mathbb{C} \rightarrow \mathbb{C}$ be a continuous SD map. Then $f$ must be the identity or the conjugate map. 
\end{theorem}

\begin{proof}
    The proof involves considering the two cases from Proposition \ref{Cbasics}.
    Since $f$ is continuous, it follows from Lemma~\ref{Lcont} that $f$ is $\mathbb{R}$ preserving. Hence, we have two cases:
    \begin{enumerate}
        \item $f(z)^2 = z^2 \implies f(z) = \pm{z} \implies f(z)/z = \pm{1}$ for all $z \in \mathbb{C}^\times $. 
        
        \noindent The continuity of $f$ implies that $\frac{f(z)}{z}$ is continuous on $\mathbb{C}^\times$ and takes values in the discrete set $\{-1, 1\}$. Furthermore, $\mathbb{C}^\times$ is a connected set and $f(1)/1 = 1$. 
        
        \noindent Hence, $\frac{f(z)}{z}$ is identically equal to 1 over $\mathbb{C}^\times$, 
        which implies that $f(z) = z$ for all $z \in \mathbb{C}^\times$.

    \item $f(z)^2 = \bar{z}^2 \implies f(z) = \pm{\bar{z}} \implies f(z)/\bar{z} = \pm{1}$ for all $z \in \mathbb{C}^\times$.

 \noindent Using the continuity of $f$ and the conjugation function $z \to \bar{z}$, we argue as above to show that $f(z) = \bar{z}$ for all $z \in \mathbb{C}^\times$.
    \end{enumerate}

\noindent Now the fact that $f(0) = 0$ completes the proof.
\end{proof}

However, we need not assume $f$ is continuous and, in fact, $f$ being $\mathbb{R}$ preserving suffices for the result above. 

\begin{theorem}\label{CRpres}
  Let $f \colon \mathbb{C} \rightarrow \mathbb{C}$ be an $\mathbb{R}$ preserving SD map. Then $f$ must be the identity or the conjugate map.  
\end{theorem}

In light of Lemma~\ref{Lcont}, it is clear that Theorem~\ref{Ccont} is a
corollary of Theorem~\ref{CRpres}. However, we decided to present the
above proof of Theorem~\ref{Ccont} as it utilizes a tool not used
elsewhere: the connectedness of $\mathbb{C}^\times$. This contrasts with
all other proofs in this work, which use the denseness of $\mathbb{Q}$ or
of $\mathbb{Q}(i)$.

\begin{proof}[Proof of Theorem~\ref{CRpres}]
 Since $f$ is $\mathbb{R}$ preserving, $f$ fixes $\mathbb{R}$.
Now let $z \in \mathbb{C}^\times$. Then $z = |z| e^{i\theta}$ where $|z| \in \mathbb{R}$ and $e^{i\theta} \in S^1$, the unit circle in the complex plane.
Hence, \[ f(z) = f(|z| e^{i\theta}) = f(|z|) f(e^{i \theta}) = |z| f(e^{i\theta}). \]
Now we consider the two cases from Proposition \ref{Cbasics}.

\begin{enumerate}
    \item Suppose $f(z/\bar{z}) = z/\bar{z}$ for all $z \in \mathbb{C}^\times$. Let $e^{i \theta} \in S^1 $ and let $w = e^{i \frac{\theta}{2}}$.
Then \[f(e^{i \theta}) = f\left(\frac{w}{\bar{w}}\right) = \frac{w}{\bar{w}} = e^{i \theta}. \]
    Hence, $f(z) = |z| f(e^{i \theta}) = |z| e^{i \theta} = z$ for all $z \in \mathbb{C}^\times.$

    \item Suppose $f(z/\bar{z}) = \bar{z}/z$ for all $z \in \mathbb{C}^\times$. Let $e^{i \theta} \in S^1 $ and let $w = e^{i \frac{\theta}{2}}$.
Then \[f(e^{i \theta}) = f\left(\frac{w}{\bar{w}}\right) = \frac{\bar{w}}{w} = e^{-i \theta}. \]
    Hence, $f(z) = |z| f(e^{i \theta}) = |z| e^{-i \theta} = \bar{z}$ for all $z \in \mathbb{C}^\times.$
\end{enumerate}
Now the fact that $f(0) = 0$ completes the proof in either case.
\end{proof}

The result above is remarkable because the same conclusion holds for $\mathbb{R}$ preserving field automorphisms over $\mathbb{C}$. In fact, it is easily verified that field automorphisms satisfy our functional equation and are, therefore, SD maps.

Automorphisms of $\mathbb{C}$ which do not preserve $\mathbb{R}$ are extremely wild -- not continuous, not even measurable, and not describable explicitly. Their existence relies on the Axiom of Choice and arises from model-theoretic constructions; see, e.g., \cite{Kestelman}, \cite{Yale}.
Hence, we cannot expect SD maps over $\mathbb{C}$ to exhibit nice properties without imposing additional conditions.

We conclude this section with another result similar to the result for Euclidean fields in the previous section.
\begin{proposition}\label{Fbasics}
    Let $\mathbb{E}$ be an Euclidean subfield of $\mathbb{R}$ and let $\mathbb{F} = \mathbb{E}(i)$. Then:
    \begin{enumerate}
        \item If $e^{i \theta} \in \mathbb{F}$, then $e^{i \frac{\theta}{2}} \in \mathbb{F}$.
        \item $\mathbb{F}$ is closed under square roots.
        \item $\mathbb{F}$ is the quadratic closure of $\mathbb{E}$.
    \end{enumerate}
\end{proposition}

\begin{proof}
  We prove each result in sequence.
  \begin{enumerate}
      \item If $e^{i\theta} = \cos \theta + i \sin  \theta \in \mathbb{F}$, then $\cos\theta, \sin\theta \in \mathbb{E} $.  Now $e^{i \frac{\theta}{2} } = \cos (\theta/2) + i \sin (\theta/2)$, with $\cos (\theta/2) = \pm \sqrt{ (1 + \cos\theta)/2}$ and $ \sin (\theta/2) = \pm \sqrt{ (1 - \cos\theta)/2}$. Since $\mathbb{E}$ is Euclidean, $\cos(\theta/2), \sin (\theta/2) \in \mathbb{E} $.
      Hence, $e^{i \frac{\theta}{2}} \in \mathbb{F}$.
      
      \item Let $z = a + ib \in \mathbb{F}$ where $a, b \in \mathbb{E}$. Note that $z = |z| e^{i\theta}$ where $\theta = \text{arg}(z)$.  Now $|z| = \sqrt{a^2 + b^2} \in \mathbb{E} \implies e^{i\theta} = z/|z| \in \mathbb{F} $. Then $w = \sqrt{|z|} e^{i \frac{\theta}{2}} $ satisfies $w^2 = z$.  Furthermore, $\sqrt{|z|} \in \mathbb{E}$, since $\mathbb{E}$ is Euclidean and $e^{i \frac{\theta}{2}} \in \mathbb{F}$ by the result above. Hence, $\pm w \in \mathbb{F}$ are the square roots of $z$.

      \item Finally, consider the quadratic equation $az^2 + bz + c = 0$ with $a, b, c \in \mathbb{F}$. By the above, $\sqrt{b^2 - 4ac} \in \mathbb{F}$, which implies that all roots of the equation lie in $\mathbb{F}$.  But if $K$ is any quadratically closed field containing $\mathbb{E}$, then $i \in K \implies \mathbb{F} = \mathbb{E}(i) \subseteq K$, which implies that $\mathbb{F}$ is the quadratic closure of $\mathbb{E}$. \qedhere
  \end{enumerate}
\end{proof}

We use this to prove the final result in this section.

\begin{theorem}
    Let $\mathbb{F}$ be the quadratic closure of an Euclidean subfield $\mathbb{E}$ of $\mathbb{R}$ and let $f\colon \mathbb{F} \to \mathbb{F}$ be an $\mathbb{R}$ preserving SD map. Then $f$ must be the identity or the conjugate map.
\end{theorem}

An example of such an $\mathbb{F}$ is the field of Algebraic Numbers --
the algebraic closure of $\mathbb{Q}$ in $\mathbb{C}$.

\begin{proof}
   By Proposition \ref{Fbasics}, $\mathbb{F} = \mathbb{E}(i)$. Since $f$ preserves $\mathbb{R}$, we know that $f(x) = x$ for all $x \in \mathbb{E}$. (See Theorem \ref{Euclidean} and the remark after the proof.)
   Hence, if $\mathbb{F}^\times$ is the multiplicative group of $\mathbb{F}$, the same reasoning as Proposition \ref{Cbasics}, gives us two possible cases.
   \begin{enumerate}
       \item $f(z/\bar{z}) = z/\bar{z}$ for all $z \in \mathbb{F}^{\times}$.
        \item $f(z/\bar{z}) = \bar{z}/z$ for all $z \in \mathbb{F}^{\times}$.
   \end{enumerate}
   (Since $\mathbb{F} = \mathbb{E}(i)$, $\bar{z}$ exists for all $z \in \mathbb{F}$.) 
   Now any $z \in \mathbb{F}^{\times}$ may be written as $z = |z| e^{i\theta}$ where $\theta = \arg(z)$. Furthermore, from Proposition~\ref{Fbasics}, $|z| \in \mathbb{E}$, $e^{i\theta} \in \mathbb{F}$ and $w = e^{i \frac{\theta}{2}} \in \mathbb{F}$.  Then
   \[f(z) = f(|z|e^{i \theta}) = f(|z|)f(e^{i \theta}) = |z| f(e^{i \theta}) \]
   because $f$ fixes $\mathbb{E}$.
 Using all the above, we consider the two cases:
   \begin{enumerate}
      \item Suppose $f(z/\bar{z}) = z/\bar{z}$ for all $z \in \mathbb{F}^\times$. Let $e^{i \theta} \in \mathbb{F} $ and let $w = e^{i \frac{\theta}{2}}$.
Then \[f(e^{i \theta}) = f\left(\frac{w}{\bar{w}}\right) = \frac{w}{\bar{w}} = e^{i \theta}. \]
    Hence, $f(z) = |z| f(e^{i \theta}) = |z| e^{i \theta} = z$ for all $z \in \mathbb{F}^\times.$

    \item Suppose $f(z/\bar{z}) = \bar{z}/z$ for all $z \in \mathbb{F}^\times$. Let $e^{i \theta} \in \mathbb{F} $ and let $w = e^{i \frac{\theta}{2}}$.
Then \[f(e^{i \theta}) = f\left(\frac{w}{\bar{w}}\right) = \frac{\bar{w}}{w} = e^{-i \theta}. \]
    Hence, $f(z) = |z| f(e^{i \theta}) = |z| e^{-i \theta} = \bar{z}$ for all $z \in \mathbb{F}^\times.$
\end{enumerate}
 Now the fact that $f(0) = 0$ completes the proof in either case.
\end{proof}

\section{Number Fields}

\noindent We now ask if there are other non-real subfields of $\mathbb{C}$ for which the only SD maps are $z$ and $\overline{z}$ without additional assumptions. The
answer is positive, and to show these results, we first provide a novel,
powerful approach. This approach is algebraic, and works well in algebraic number fields,
i.e.\ finite extensions of $\mathbb{Q}$ (which in particular are
finite-dimensional $\mathbb{Q}$-vector spaces). We illustrate this approach with the simplest class of number
fields: quadratic extensions of $\mathbb{Q}$ -- which include
$\mathbb{Q}(\sqrt{2})$ and the Gaussian rationals $\mathbb{Q}(i)$.

We begin with a key lemma, which, loosely speaking, states that if an SD map fixes two consecutive terms and the common difference of an arithmetic progression, then it fixes all the terms. 

\begin{lemma}[AP Lemma]\label{Llattice}
Suppose $\mathbb{F}$ is a field of characteristic $0$ and let $\{ \cdots, a, a + d, a+2d, \dots \}$ be an
arithmetic progression in
$\mathbb{F}$ that does not contain $0$. If $f \colon \mathbb{F} \to
\mathbb{F}$ is an SD map that fixes $a, a+d$, and  $d$, then $f$ fixes all terms of this progression.
\end{lemma}

\begin{proof}
The result is obvious when $d =0$. So assume $d \ne 0$.  By Theorem~\ref{Tmult}, $f$ is injective and multiplicative on
$\mathbb{F}$. Now setting $x = a+d$ and $y = d$ in~\eqref{Esumdiff}, we get:
\[
f \left( \frac{a+2d}{a} \right) = \frac{f(a+d) + f(d)}{f(a+d)-f(d)} =
\frac{(a+d) + (d)}{(a+d)-(d)} =
\frac{a+2d}{a},
\]
where the denominators are nonzero because $f$ 
injective. Now, by multiplicativity, and using the fact that $f(a) =a$, we get $f(a+2d) = a+2d$. One can now proceed ``forward'' inductively. 

The backward direction is similar: set $x=a$ and $y=-d$. Since $f$ is an odd function,
\[f \left(\frac{a-d}{a+d} \right) = \frac{f(a)+f(-d)}{f(a)+f(d)} = \frac{f(a) - f(d)}{f(a)+f(d)} =\frac{a-d}{a+d}\]
Using multiplicativity of $f$ and the fact that $f(a+d)= a+d$, we get that $f(a-d) = a-d$. Similarly, one can proceed ``backward'' inductively. 
\end{proof}

\begin{proposition}\label{Psunil}
Suppose  $d \in \mathbb{Q}$ is such that
$\sqrt{d} \not\in \mathbb{Q}$. 
Define the conjugation automorphism on
$\mathbb{Q}(\sqrt{d})$ via: $\overline{a + b \sqrt{d}} := a - b\sqrt{d}$ for $a,b \in \mathbb{Q}$.
Let $f$ be an SD map on $\mathbb{Q}(\sqrt{d})$. Then exactly one of the following happens:
\begin{enumerate}
\item $f(\sqrt{d}) = \sqrt{d}$, and for any $z$, $f(z) = \pm z$; or
\item $f(\sqrt{d}) = -\sqrt{d}$,  and for any $z$,  $f(z) = \pm  \overline{z}$.
\end{enumerate}
\end{proposition}

\begin{proof}
By Theorem~\ref{Tmult}, $f$ is multiplicative on the quadratic extension
$\mathbb{Q}(\sqrt{d})$. Next, using the hypotheses, $f(\sqrt{d})^2 = f(d)
= d$, we get $f(\sqrt{d}) = \pm \sqrt{d}$. Now there are two cases.
\begin{enumerate}
\item $f(\sqrt{d}) = \sqrt{d}$. Then $f(y'\sqrt{d}) = y'\sqrt{d}$ for all
$y' \in \mathbb{Q}$. Now let $x \in \mathbb{Q}$ and $y = y'\sqrt{d}$ for
$y' \in \mathbb{Q}$, in~\eqref{Esumdiff}. Then for $(x,y') \neq (0,0)$ we
have:
\[
f \left( \frac{x+y'\sqrt{d}}{x-y'\sqrt{d}} \right) =
\frac{f(x)+f(y'\sqrt{d})}{f(x)-f(y'\sqrt{d})} =
\frac{x+y'\sqrt{d}}{x-y'\sqrt{d}}.
\]

Now let $z = x + y'\sqrt{d} \in \mathbb{Q}(\sqrt{d})^\times$ and define
$|z| := z \overline{z} = x^2 - d(y')^2 \in \mathbb{Q}^\times$ (and define
$|0| := 0$). The multiplicativity and ``fixing of $\mathbb{Q}$'' by $f$
yields:
\begin{equation}\label{Esunil}
\frac{1}{|z|^2} f(z)^2 =
\frac{1}{|z|^2} f(z^2) = f(z^2/|z|^2) = f(z/\overline{z}) =
z/\overline{z} = z^2/|z|^2.
\end{equation}
Hence $f(z)^2 = z^2$ for all $z$, including $z=0$. This means, for any $z$,  $f(z) = \pm z$.

%Moreover,
%\[
%\varepsilon(z) |z|^2 = f(z) \overline{z} = f(\overline{z}) z =
%\varepsilon(\overline{z}) |z|^2,
%\]
%and since $z \neq 0$, we get $\varepsilon(z) =
%\varepsilon(\overline{z})$, as desired.

\item $f(\sqrt{d}) = -\sqrt{d}$. This case is similar: $f(y'\sqrt{d}) = -y'\sqrt{d}$ for all $y' \in
\mathbb{Q}$, and the above choice of $y = y'\sqrt{d}$ and $x$, with $z = x + y' \sqrt{d}$ yields:
\[
\frac{1}{|z|^2} f(z)^2 = f(z/\overline{z}) = \overline{z}/z =
\overline{z}^2 / |z|^2.
\]

Hence $f(z)^2 = \bar{z}^2$ for all $z$, including $z=0$. This means, for any $z$,  $f(z) = \pm \bar{z}$.
\end{enumerate}
\end{proof}

We now prove the main theorem of this section.

\begin{theorem}[All quadratic fields]\label{Tquad}
Suppose $d$ is an integer with $\sqrt{d} \not\in \mathbb{Q}$. If $f
\colon \mathbb{Q}(\sqrt{d}) \rightarrow  \mathbb{Q}(\sqrt{d}) $ is an SD
map, then
$f$ is either the identity or the conjugate map.
\end{theorem}

\begin{proof}
Any element in $ \mathbb{Q}(\sqrt{d})$ can be written as $r(m+n\sqrt{d})$ where $r \in \mathbb{Q}$ and $m$ and $n$ are integers. Since our SD map $f$ is multiplicative and fixes rationals, it is enough to prove the result for elements of the form $m + n \sqrt{d}$, where both $m$ and $n$ are integers. We apply Proposition~\ref{Psunil}, and consider the two possibilities
separately:
\begin{enumerate}
\item $f(\sqrt{d}) = \sqrt{d}$: By Proposition~\ref{Psunil}, $f(1+\sqrt{d}) = \pm (1 + \sqrt{d})$. We will show that $f(1+\sqrt{d}) =  (1 + \sqrt{d})$.
Suppose for contradiction that $f(1+\sqrt{d}) = -(1+\sqrt{d})$. From
this and $f(\sqrt{d}) = \sqrt{d}$, by Lemma~\ref{Llattice} we get
\begin{align*}
f(2+\sqrt{d}) = &\ 
\sqrt{d} \cdot f \left( \frac{2 + \sqrt{d}}{\sqrt{d}} \right) = \sqrt{d} \cdot f \left( \frac{(1 + \sqrt{d}) + 1}{(1 + \sqrt{d}) - 1} \right)\\
= &\ \sqrt{d} \cdot \frac{-(1+\sqrt{d}) + 1}{-(1+\sqrt{d})-1}
= \frac{d}{2+\sqrt{d}}.
\end{align*}

If this equaled $\pm (2+\sqrt{d})$, then
$0 = (2+\sqrt{d})^2 \pm d$, so $4\sqrt{d} =  -4 - d \mp d
\in \mathbb{Q}$,
which is false. Thus $f(2+\sqrt{d}) \neq \pm (2+\sqrt{d})$, which
contradicts Proposition~\ref{Psunil}.

Now, we repeatedly apply our AP lemma to show that $f$ fixes $m+n \sqrt{d}$ for all integers $m$ and $n$. Place each $m+n \sqrt{d}$ at the corresponding lattice point $(m, n)$ and note that $f$ fixes elements on the coordinate axes. Since $1+\sqrt{d}$ and $\sqrt{d}$ are fixed, $f$ fixes all elements at the vertical line $x=1$. To see that all elements along each horizontal line  $y=n$ are fixed, apply the AP lemma for the arithmetic progression with consecutive terms $n \sqrt{d}$ and $1+ n \sqrt{d}$. 

\item $f(\sqrt{d}) = -\sqrt{d}$: we claim that $f(1+\sqrt{d}) =
1-\sqrt{d}$. This is  proved as above: if not, then $f(1+\sqrt{d}) =
\sqrt{d}-1$; now
\[
f(2+\sqrt{d}) = -\sqrt{d} \cdot \frac{\sqrt{d}-1+1}{\sqrt{d}-1-1} =
\frac{-d}{\sqrt{d}-2}.
\]
By Proposition~\ref{Psunil}, this equals $\pm (2-\sqrt{d})$, so we
again get $4 \sqrt{d} \in \mathbb{Q}$, which is false 
The rest of the proof is again similar to the above, using the AP lemma applied to the map $\bar{f}$; we leave this as an exercise to the reader.
\qedhere

\end{enumerate}
\end{proof}

Theorem~\ref{Tquad} is a strong result with several consequences worth mentioning explicitly.
For instance, it implies, without any additional assumptions, that the only SD maps over $\mathbb{Q}(\sqrt{d}) \subset \mathbb{R}$ are precisely the only possible field automorphisms, $z$ and $\bar{z}$. (In fact, these are the two maps on $\mathbb{Q}(\sqrt{2})$ that we alluded to  immediately after Corollary \ref{Qpcont}.)

Moreover, this also covers all quadratic extensions of
$\mathbb{Q}$,\footnote{To see why: all such extensions are $\mathbb{Q}[X]
/ (aX^2 + bX + c)$ with $a,b,c$ rational and $a \neq 0$. Thus, we 
attach to $\mathbb{Q}$ a root $\alpha$ of the rational quadratic
polynomial $X^2 + (b/a)X + (c/a)$, i.e.\ of $(X + (b/2a))^2 -
\frac{b^2-4ac}{4a^2}$. I.e., we attach $\sqrt{p/q}$
where $p,q \in \mathbb{Z}$ and $p/q = b^2-4ac$. This is equivalent to
attaching $\sqrt{d}$, where $d=pq$.}
including the ``Gaussian fields'' $\mathbb{Q}(i)$ as well as the
cyclotomic field $\mathbb{Q}(\zeta_3)$, where $\zeta_3 = \frac{-1 + i
\sqrt{3}}{2}$ is a primitive cube root of unity (so $d = -3$).

Finally, if we add a continuity condition, the following result is immediate.

\begin{corollary}
    Let $\mathbb{F}$ be a field such that $\mathbb{Q}(i) \subseteq \mathbb{F} \subseteq \mathbb{C}$ and let 
    $f\colon \mathbb{F} \to \mathbb{F}$ be a continuous SD map. Then $f$ must be the identity or the conjugate map.
\end{corollary}

\section{SD maps that are not automorphisms}
At this point, it is natural for the reader to wonder whether field automorphisms are the only
solutions to our functional equation.
After all, pure algebraic manipulation of the functional equation showed that all solutions
must be injective and multiplicative. Could further manipulation perhaps yield surjectivity
and additivity as well?

Note that if this were the case, then the result would hold over all
fields. However, the following example shows that unlike injectivity,
surjectivity need not follow from the functional equation.

\begin{example}
Let $\mathbb{K}$ be any field, and consider its transcendental extension:
\[
\mathbb{F} := \mathbb{K}(x_0, x_1, x_2, \dots).
\]
Define a map $ f : \mathbb{F}
\to \mathbb{F} $ that fixes all elements of $\mathbb{K}$ and  sends $x_n \mapsto x_{n+1}$ for all $n \geq 0$. 
This is a field monomorphism, and hence an SD map. However, this is not an automorphism for any base
field $\mathbb{K}$ because it is not surjective.
\end{example}

Here is yet another example of an SD map that is not surjective. The field here is, in fact, a subfield of $\mathbb{R}$.

\begin{example}
Consider the subfield $\mathbb{F} = \mathbb{Q}(\pi)$ -- since $\pi$ is
transcendental over $\mathbb{Q}$ -- and for $k \in \mathbb{Z}$ define the
map $f_k : \mathbb{F} \to \mathbb{F}$ by
\begin{equation}\label{Efk}
f_k : \frac{p(\pi)}{q(\pi)} \mapsto \frac{p(\pi^k)}{q(\pi^k)}, \qquad
p,q \in \mathbb{Q}[x].
\end{equation}
Then $f_2, f_3, \dots$ are all SD maps, but not surjective.
\end{example}

The above, of course, leaves open the question of whether automorphisms are the only
surjective SD-maps over a field. However, rather than imposing surjectivity artificially, it would be more interesting to
consider SD-maps over a finite field, where injectivity automatically implies surjectivity.
So, are automorphisms the only SD-maps over a finite field? To our surprise, even that is not true as shown in the example below.

\begin{example}
Consider $\mathbb{F}_5$ -- the field of $5$ elements. Note that if $f$ is an SD map on this field, then the restriction of $f$ to its multiplicative group $\mathbb{F}_5^\times \cong C_4$ must be an automorphism. But what are automorphisms of $C_4$? It is either $x \mapsto x$ or $x \mapsto x^3$. The identity map is clearly an automorphism of $\mathbb{F}_5$. What about the $x \mapsto x^3$ map on $\mathbb{F}_5$? We leave it as an amusing exercise to the reader to verify that this map is not an automorphism, but an SD map on $\mathbb{F}_5$!
\end{example}

The above examples hint at something deeper underlying our functional
equation and raise several natural questions. The final example,
especially, provides a strong motivation to explore SD maps over fields
of characteristic $p$. What, if anything, is special about the field with
five elements? Can we classify  SD maps over other finite fields? Are
field automorphisms the only SD maps over algebraic extensions of finite
fields? These questions point to a broader and richer theory, which we
will address in a forthcoming sequel \cite{CKS}.

\subsection*{Concluding remarks.}

We end on a philosophical note. The Cauchy functional
equation~\eqref{Ecauchy} is a hundred years old and has seen much study.
Nowadays, one frequently encounters functional equations in mathematics contests -- where they are typically solved via ingenious algebraic manipulations. %(On that note: perhaps~\eqref{Esumdiff} has also been seen there?)
However, solving~\eqref{Ecauchy} quickly led mathematicians to a fruitful study of deeper
topics in analysis such as continuity, boundedness, and measurability.

Similarly, in this work, we see many different proof-philosophies
emerging from our equation~\eqref{Esumdiff} whose full ramifications go
beyond merely solving~\eqref{Esumdiff}).
The first is algebraic: showing via clever manipulations
that~\eqref{Esumdiff} implies fixing $\mathbb{Q}$, via solving the cubic
$f(8) = f(2)f(4)$ in a field.
The second relates to analysis: a ``freshman calculus'' level density argument and ordering inside $\mathbb{R}$ (or the real/complex constructible or algebraic numbers in
it).
The third is topological, in the case of continuous solutions over $\mathbb{C}$. Finally, a fourth approach using arithmetic progressions is useful in number fields like
$\mathbb{Q}(\sqrt{d})$.
Thus, the functional equation~\eqref{Esumdiff} leads one to explore
a rich body of results in multiple branches of mathematics.

%\vfill\eject

\end{document}